\documentclass[11pt]{article}
\usepackage{geometry} 
\usepackage{amsmath,amsthm,amssymb}                
\usepackage{graphicx, color}
\usepackage{amssymb}
\usepackage{epstopdf}
\usepackage{pdfsync}
\usepackage{mathabx}

\newtheorem{definition}{Definition}[section]
\newtheorem{lemma}[definition]{Lemma}
\newtheorem{theorem}[definition]{Theorem}
\newtheorem{proposition}[definition]{Proposition}
\newtheorem{corollary}[definition]{Corollary}
\newtheorem{remark}[definition]{Remark}

\RequirePackage{authblk}

\numberwithin{equation}{section}

\def\e{\varepsilon}

\def\dx{\,dx}

\title{Asymptotic analysis of higher-order perturbations\\
of the Perona--Malik functional}
\author[1]{Andrea Braides} 
\author[2]{Irene Fonseca}
\affil[1]{\small University of Rome Tor Vergata, Rome, Italy, {\tt braides@mat.unroma2.it}}
\affil[2]{\small Carnegie Mellon University, Pittsburgh, USA, {\tt fonseca@andrew.cmu.edu}}
\date{}
\begin{document} 

\maketitle

\begin{abstract}
The $\Gamma$-limit of higher-order singular perturbations of the Perona–-Malik functional is analyzed. The energies considered combine the critically scaled logarithmic term with a $k$-th order regularization designed to balance bulk and interfacial effects. A compactness result is obtained, and the $\Gamma$-limit is identified as a free-discontinuity functional on SBV, given by the sum of the Dirichlet energy and a surface term proportional to the jump amplitude to the power 
$1/k$. The surface density is characterized through a one-dimensional optimal-profile problem with homogeneous boundary conditions on derivatives up to order $k-1$. As a consequence, the limit of the same energies at a different scaling is determined. That scaling had been previously studied in the second-order case to address the so-called staircasing phenomenon.



\end{abstract}
\section{Introduction}
The Perona--Malik equation in Image Processing is a formally ill-posed equation, which can nevertheless be discretized or smoothed to obtain stable numerical schemes \cite{56205,perona1994anisotropic,MR1470926}. From the standpoint of the Calculus of Variations, such an equation is related to the functional
$$
\int_\Omega \log(1+|\nabla u|^2)dx,
$$
which is a non-convex functional whose minimizers tend to develop singularities. In terms of the direct methods of the Calculus of Variations, this is highlighted by the fact that its lower-semicontinuous envelope is the functional identically $0$. 
Nevertheless, its form suggests an interesting behavior of minimizing sequences for minimum problems involving this energy. Indeed, the functional is very close to the Dirichlet integral for small values of the gradient, which suggests a smoothening effect, while it is non-convex for high values of the gradient, which suggests the formation of interfaces. This analysis can be formalized in different ways. One way is by discretization using finite differences, introducing a cubic mesh and approximating the integral with a finite-difference sum. By properly scaling the finite differences and the mesh size, one can transform this analysis in the computation of a $\Gamma$-limit in a discrete-to-continuum setting allowing for surface energies (see \cite{MR4670293} for a general introduction). In particular, a result by Morini and Negri \cite{M-N} shows that the suitably scaled discretizations of the Perona--Malik energy are approximated by (an anisotropic version of) the Mumford--Shah functional \cite{MR997568}, within the theory of free-discontinuity functionals introduced by De Giorgi \cite{BLN}.

A different approach to the analysis of the Perona--Malik energy is by adding a singular perturbation.
This has been shown to give rise to free-discontinuity problems with a surface energy depending on the jump size. The effects of this surface contribution can be analyzed to explain {\em staircasing} phenomena \cite{GP-2,MR4902836} (see also \cite{MR4749793} for the analysis of the discretized version). 
In this work, we consider a perturbation of the Perona--Malik functional of the form
\begin{equation}\label{PMe}
F_\e(u)\coloneq\int_\Omega \frac{1}{\e|\log\e|} \log(1+\e|\log\e||\nabla u|^2)\dx+ \e^{2k-1} \int_\Omega \|\nabla^{(k)}u\|^2\dx,
\end{equation}
where $\|\cdot\|$ denotes the operator norm, defined on $H^k(\Omega)$. The scalings are designed so that in the limit both bulk and surface energies appear. Hence, this scaling gives a complete picture of the behaviour of $F_\e$.
Heuristically, the first integral  in \eqref{PMe} pointwise converges to the Dirichlet integral  as $\e\to 0$; however, thanks to the sub-linear growth of the logarithm, it may be finite (and even infinitesimal) on sequences $\{u_\e\}$ converging to a function with jump discontinuities or non-trivial Cantor part of the derivative. The addition of the singular perturbation prevents the concentration on Cantor derivatives, and gives a finite energy on jump discontinuities. In the language of $\Gamma$-convergence, we will show that the $\Gamma$-limit of $F_\e$ in the $L^1$-topology is given by
\begin{equation}\label{limform}
F(u)\coloneq\int_\Omega |\nabla u|^2\dx+ m_k\int_{S(u)\cap\Omega}\root k\of{|u^+-u^-|}d\mathcal H^{N-1},
\end{equation}
on special functions of bounded variation. The appearance of a $k$-th power has been conjectured by Gobbino and Picenni \cite{GP-2}, and can be seen as a consequence of the scaling properties of the singular perturbation with the $k$-th derivative. 
The constant $m_k$ in \eqref{limform} is given by
\begin{eqnarray}\label{emmeka}\nonumber
&&m_k\coloneq\inf_{T>0}\min\bigg\{T+\int_0^T|v^{(k)}|^2\,dt: v\in H^k(0,T),  v(0)=0, v(T)=1, \\ &&\hskip4cmv^{(\ell)}(T)=v^{(\ell)}(0)=0 \hbox{ for all } \ell\in\{1,\ldots,k-1\}\Big\}.
\end{eqnarray}
In this formula, we optimize the sum of a length and a transition energy on an interval of that length given by the $k$-th derivative with homogeneous conditions on the other derivatives. 

In order to explain the nature of the constant $m_k$, we refer to a recent paper by Solci \cite{S-2}, in which the author studies functionals of the form
\begin{equation}\label{ABGe}
F^S_\e(u)\coloneq\int_\Omega \min\Big\{|\nabla u|^2,\frac1\e\Big\}\dx+ \e^{2k-1} \int_\Omega \|\nabla^{(k)}u\|^2\dx.
\end{equation}
In comparison to that case, the functionals $F_\e$ in 
\eqref{PMe}, instead, are shown to satisfy a lower bound, of the form
\begin{equation}\label{weaklb}
F_\e(u)\ge\int_\Omega \min\Big\{\e^{1-2p_\e}|\nabla u|^2,\frac1\e\Big\}\dx+ \e^{2k-1} \int_\Omega \|\nabla^{(k)}u\|^2\dx,
\end{equation}
for some $p_\e$ slowly tending to $0$. This is a much weaker estimate than that for $F^S_\e$.
The main novelty of the paper is in proving that suitable adaptations of the local interpolation methods of \cite{S-2} allow to work in the low-coerciveness framework \eqref{weaklb}. As a consequence, we can estimate the energy necessary to create a jump discontinuity by minimizing on functions with gradient `above the threshold'; that is, where $\e^{1-2p_\e}|\nabla u|^2>1/\e$. Reducing to one-dimensional profiles, we then prove that $m_k$ has the form \eqref{emmeka}. We note that this form is the same as for the limit of the functionals $F_\e^S$, and hence independent of the logarithmic growth. The additional homogeneous boundary conditions on derivatives of order $\ell\ge2$ are obtained by using local interpolation techniques. We note that if $k=2$ and we consider functionals $F^S_\e$, then the boundary conditions in \eqref{emmeka}, which are only on the first derivative,  are directly obtained from the threshold condition. This case had previously been treated in \cite{abg} and \cite{BDSepp}.
Once higher-order effects are taken into account by the jump part of the limit energy, the bulk part can be handled as for the discrete approximations of the Mumford--Shah functionals in \cite{M-N}. 

From the analysis of functionals \eqref{PMe}, we also deduce the limit of functionals
$$
\mathbb F_\e(u)\coloneq\int_\Omega \frac{1}{2\e|\log\e|} \log(1+|\nabla u|^2)\dx+ \e^{2k-1} \int_\Omega \|\nabla^{(k)}u\|^2\dx,
$$
with a different scaling of the Perona--Malik functional, analyzed by Gobbino and Picenni \cite{GP-2} when $k=2$ (see also previous work by Bellettini and Fusco \cite{MR2403710}). In this case, we show that the limit is
$$
\mathbb F(u)\coloneq m_k\int_{S(u)\cap\Omega}\root k\of{|u^+-u^-|}d\mathcal H^{N-1}
$$
with domain functions in $SBV(\Omega)$ with $\nabla u=0$ almost everywhere.

\bigskip
The structure of the paper is as follows.  After stating the main result and remarking that we can reduce to studying the one dimensional case, in Section \ref{se-pre} we gather the relevant results preliminary to the identification of the $\Gamma$-limit. After proving that the weak lower bound \eqref{weaklb} holds, we proceed and prove a number of lemmas that allow us to treat functions whose gradient exceeds the threshold $|\nabla u|=\e^{p_\e}/\e$ in \eqref{weaklb}. Loosely speaking, every interval `above the threshold' corresponds to a limit jump, and minimizing on functions corresponding to the same jump gives the limit jump energy density. The main issue here, when $k>2$, is the determination of optimal boundary conditions, which are proved to be homogeneous on all derivatives. This is achieved by proving that if the jump is sufficiently large, then intervals above the threshold can be slightly extended in such a way that all derivatives are small  at the boundary.  
In Section \ref{ga-sec} we then study the $\Gamma$-limit, on the one hand, using the preliminary results to characterize the limit jump energy, and on the other hand using the properties of the scaling of the logarithmic part in \eqref{PMe} to prove that the bulk part of the limit is a Dirichlet integral. Finally, in Section \ref{sufi} we address the functionals $\mathbb F_\e$.

\section{Statement of the convergence results}
We recall that $SBV(\Omega)$ is the subset of the function $BV(\Omega)$ whose distributional derivative $Du$ can be written as $Du= \nabla u\, \mathcal L^d+ (u^+-u^-)\nu\mathcal H^{d-1}\lfloor S(u)$, where $\nabla u$ is the approximate differential of $u$, $S(u)$ is the set of discontinuity points of $u$, $\nu$ is the measure-theoretical normal to $S(u)$ and $u^\pm$ are the approximate values of $u$ on both sides of $S(u)$. The space $GSBV(\Omega)$, furthermore, is the set of all measurable functions $u$ whose truncations $(u\vee T)\wedge(-T)$ are in $SBV(\Omega)$ for all $T>0$. We refer to \cite{MR1857292,BLN} for an introduction to these spaces, and a rigorous definition of these objects. 

In our context, we will mainly deal with the one-dimensional case, in which a function $u$ belongs to $SBV(0,1)$ if $u=v+w$, where $v\in W^{1,1}(0,1)$ is absolutely continuous and $w=\sum_{i\in I} z_i\chi_{(t_i,1)}$, where $I$ is a finite or countable set of indices and $z_i\neq 0$ are such that $\sum_i|z_i|<+\infty$, is a function with only jump discontinuities, so that it is not unreasonable to think of $SBV$ functions as piecewise $W^{1,1}$ functions.  Moreover, since in our context the size of the discontinuities is bounded on functions with finite $\Gamma$-limit, the domain of the latter is indeed contained in $SBV(0,1)$ and not only in $GSBV(0,1)$ in the one-dimensional case.

\bigskip
The main result of the paper is the following theorem.

\begin{theorem}\label{main}Let $\Omega$ be a bounded open set in $\mathbb R^d$ and let
\begin{equation}
F_\e(u)\coloneq\int_\Omega \frac{1}{\e|\log\e|} \log(1+\e|\log\e||\nabla u|^2)\dx+ \e^{2k-1} \int_\Omega \|\nabla^{(k)}u\|^2\dx,
\end{equation}
be defined on $H^k(\Omega)$,
where $\|\cdot\|$ denotes the operator norm of a $k$-th order tensor. Then for every sequence $u_\e$ such that $F_\e(u_\e)\le S<+\infty$ there exist a subsequence and constants $C_\e$ such that $u_\e+C_\e$ converge in measure to a function $u\in GSBV(\Omega)$. Furthermore, the $\Gamma$-limit of $F_\e$ in the $L^1$-topology and in measure is given by
\begin{equation}
F(u)\coloneq\int_\Omega |\nabla u|^2\dx+ m_k\int_{S(u)\cap\Omega}\root k\of{|u^+-u^-|}d\mathcal H^{N-1},
\end{equation}
with the constant $m_k$ given by
\begin{eqnarray}\label{emmeka-2}\nonumber
&& m_k\coloneq\inf_{T>0}\min\bigg\{T+\int_0^T|v^{(k)}|^2\,dt: v\in H^k(0,T),  v(0)=0, v(T)=1, \\ &&\hskip4cm v^{(\ell)}(T)=v^{(\ell)}(0)=0 \hbox{ for all } \ell\in\{1,\ldots,k-1\}\Big\}.
\end{eqnarray}
\end{theorem} 
We note that the constant $m_k$ can be actually computed since it amounts to finding the coefficients of the polynomial of degree $2k-1$ symmetric with respect to $(T/2, 1/2)$ and subjected to the given boundary conditions, and then a one-dimensional optimization of $T$. However, formula \eqref{emmeka-2} is useful since it highlights that a jump corresponds to an optimal-profile problem, and suggests the form of the recovery sequences close to the jump points of the limit. 

\smallskip
This result will be proved in dimension one, the extension to higher dimension being briefly described in the following section.

\subsection{Reduction to dimension one}

As remarked in \cite{S-2} for the truncated quadratic potential, we can reduce to dimension one by a slicing and approximation argument.

1) if $u^{x,\nu}$ indicates the one-dimensional section of $u$ given by $u^{x,\nu}(t)=u(x+t\nu)$,
a Fubini argument ensures the coerciveness of almost all sections. Note that this is true whatever norm one takes on the tensor $\nabla^{(k)}u$;

2) the liminf inequality likewise follows from the one-dimensional result by slicing;

3) if $\mathcal H^{d-1}(S(u))<+\infty$ we can use a density argument in \cite{AG}, exploited in the case $k=2$ for the truncated quadratic potential, by functions with smooth $S(u)$. One has to use some extra care to control the   derivatives in the orthogonal direction to $S(u)$ appearing in $\nabla^{(k)}u$ for $k>2$, but this term is the same as in the analogous singular perturbation problem for double-well potentials in  \cite{BDS} and \cite{MR4955662}, and we can refer to those papers for details; 

4) finally, if $\mathcal H^{d-1}(S(u))=+\infty$ a recent result by Conti, Focardi and Iurlano \cite{MR4801506} ensures an approximation with finite measure, so that we can use point 3.

\smallskip
The slicing argument above is standard, but rather delicate. We refer to the recent PhD thesis of Picenni \cite{Picenni-thesis} for a treatment with all details.

\smallskip 

\section{Preliminaries}\label{se-pre}

We now give some preliminary results in order to provide a lower bound for $F_\e$. 
To that end, with a slicing argument in mind, we consider the one-dimensional case, for which
it suffices to treat the case
 \begin{equation}
F_\e(u)\coloneq\int_{(0,1)} \frac{1}{\e|\log\e|} (1+\e|\log\e||u'|^2)\dx+ \e^{2k-1} \int_{(0,1)} |u^{(k)}|^2\dx,
\end{equation}

\bigskip
We begin with some simple consequences of the properties of the logarithm, that can be found
e.g.~in \cite[Lemmas 3.2 and 3.3]{M-N}.

\begin{lemma}\label{MN-lemma} Let $p_\e$ be such that $p_\e\to 0$ as $\e\to 0$ and
\begin{equation}
\lim_{\e\to 0}\frac{\log|\log\e|}{p_\e|\log\e|}=0,
\end{equation} and let $c_\e\coloneq\e^{p_\e}$. Then

\smallskip
{\rm(a)} $\displaystyle\lim_{\e\to 0^+} c_\e=\lim_{\e\to 0^+}c_\e|\log\e|
=0${\rm ;}\hskip 1cm
{\rm(b)} $\displaystyle\lim_{\e\to 0^+} \frac{\log\big(1+\e|\log\e|\frac{c_\e^2}{\e^2}\big)}{|\log\e|}=1$.\end{lemma}

\bigskip

From claim (b) in Lemma \ref {MN-lemma} we obtain the following estimate, which will be fundamental in proving the lower bound for the limit of $F_\e$. It shows that some compactness and lower bounds for $F_\e$ can be obtained by looking at energies in which the logarithmic part is substituted by a truncated quadratic potential, which is constant after the threshold ${c_\e}/\e$.

\begin{proposition}\label{sub-prop}
For all $\eta>0$ fixed we have 
$$
\frac1{\e|\log\e|}\log (1+\e|\log\e|z^2)\ge (1-\eta)\max\Big\{\frac\e{c_\e^2} z^2,\frac1\e\Big\}= (1-\eta)\max\Big\{\e^{1-2p_\e} z^2,\frac1\e\Big\}
$$
for all $z\in \mathbb R$ and $\e>0$ sufficiently small.
\end{proposition}

\subsection{Estimates for sets with large derivatives below the threshold $c_\e/\e$}

In order to define a free-discontinuity energy, the heuristic argument is that to a maximal interval $(s_\e,t_\e)$ in which $|u'_\e|$ is above the threshold $\tfrac{c_\e}\e$ there corresponds a jump of size $z_\e=|u_\e(t)-u_\e(s)|$, and minimizing among a class of functions $v$ with the same jump, we obtain an expression for the energy density at this jump as $\e\to 0$. Unfortunately, the only other boundary conditions that we have are on the first derivative as $u'_\e(t)=u'_\e(s)\in \{-\tfrac{c_\e}\e,\tfrac{c_\e}\e\}$, which are not sufficient to determine this energy density. The idea is then to extend the interval $(s_\e,t_\e)$ to an interval $(\sigma_\e,\tau_\e)$, not much larger than $(s_\e,t_\e)$ and such that derivatives of $u_\e$ of all orders up to $k-1$ are small (up to scaling) at its boundary. To that end, we have to study intervals where $u_\e$ is  below threshold but its derivatives of all orders up to $k-1$ are not small (up to scaling).

In \cite{S-2} very fine interpolation estimates have been used for sequences of functions for which the energy
\eqref{ABGe} is equibounded. We will show that the arguments in \cite{S-2} can be adapted to the extreme case when  we have the truncated quadratic potential in Proposition \ref{sub-prop} with the coefficient of the quadratic part growing just slightly slower than $\e$.  

We use a notation parallel to that in \cite{S-2}, where the threshold is $1/\sqrt\e$. We fix a sequence  $u_\e$ and denote by $S$ an upper bound for $G_\e(u_\e)= G_\e(u_\e, (0,1))$, where
 \begin{equation}\label{defGI}
G_\e(u; I):=\int_{I} \min\big\{\e^{1-2p_\e}|u'|^2,\tfrac1\e\big\} \dx+ \e^{2k-1} \int_{I} |u^{(k)}|^2\dx
\end{equation}
for all $I\subset (0,1)$.
Note that, taking $\delta=\frac12$, by Proposition \ref{sub-prop} the upper bound holds if we have $F_\e(u_\e)\le \frac12S<+\infty$. In order to apply some results in \cite{S-2}, we assume that $2S^2>1$, which we can always do up to taking a higher upper bound. 

We remark that, in order to make a parallel between the properties obtained assuming the boundedness of \eqref{ABGe} and those obtained assuming the boundedness of \eqref{defGI}, that the first energy can be written as the second one with $p_\e=\frac12$. 

For all $\e>0$ we set 
$$
\mathcal A_\e:=\Big\{t\in (0,1): |u_\e^\prime(t)|<\frac{c_\e}{\e}\Big\}=\bigcup_{j\in \mathcal I_\e}I^\e_j.
$$
Since each $u'_\e$ is of bounded variation, upon slightly varying $c_\e$ without changing its properties, we can suppose that $\mathcal I_\e$ is a finite set, and moreover that there are no pairs of intervals $I^\e_j$ with a common extreme. Hence $(0,1)$ is decomposed into a finite family of open intervals of $\mathcal A_\e$ and the complement, which is a finite family of closed intervals.

Note that by Proposition \ref{sub-prop} we have 
 \begin{equation}\label{menoeps}
 |(0,1)\setminus \mathcal A_\e|=\Big|\Big\{t\in (0,1): |u_\e^\prime(t)|\geq\frac{c_\e}{\e}\Big\}\Big|\leq \e S. 
 \end{equation}

\begin{remark}[Estimates by localized interpolations]
\rm If $I\subset \mathcal A_\e$ then we have
  \begin{equation}
G_\e(u_\e; I)= \e^{1-2p_\e}\int_{I} |u_\e'|^2 \dx+ \e^{2k-1} \int_{I} |u_\e^{(k)}|^2\dx;
\end{equation}
that is, $G_\e(u_\e; I)$ can be expressed as a combination of the $L^2$ norms of $u'_\e$ and $u^{(k)}_\e$. This will allow us to estimate the $L^2$ norm of an intermediate derivative  $u^{(\ell)}_\e$ with $G_\e(u_\e; I)$
and the $L^2$ norms of $u'_\e$. More precisely, we will show that
\begin{eqnarray}\label{interp2-3}
\|u_\e^{(\ell)}\|^2_{L^2(I)}\le R_k\e^{1-2\ell-2p_\e\frac{k-\ell}{k-1}} G_\e(u_\e;I)
+R_k \frac{\|u_\e^\prime\|^2_{L^2(I)}}{|I|^{2(\ell-1)}}
\end{eqnarray}
for some explicit constant  $R_k$ depending only on $k$.

To show \eqref{interp2-3}, we begin by following the arguments of the proof of \cite[Lemma 7]{S-2} word for word. We first use the interpolation inequality
\begin{equation}\label{interp1}
 \|u^{(\ell)}\|^2_{L^2(I)}\leq R_k\|u^\prime\|^{2\theta}_{L^2(I)}\|u^{(k)}\|^{2(1-\theta)}_{L^2(I)}+R_k \frac{\|u^\prime\|^2_{L^2(I)}}{|I|^{2(\ell-1)}},
\end{equation}
 valid for $u\in H^k(I)$, 
where $\theta=\frac{k-\ell}{k-1}$ and $R_k$ is an explicit constant independent of $I$  (see \cite[Theorem 7.41]{leofrac}), and observe that for $\alpha>0$ we have
\begin{eqnarray}\label{interp2}
\e^\alpha \|u^{(k)}\|^{2(1-\theta)}_{L^2(I)} \|u^\prime\|^{2\theta}_{L^2(I)} 
\leq \e^{\alpha \frac{k-1}{\ell-1}} \|u^{(k)}\|^{2}_{L^2(I)}+\|u^\prime\|^{2}_{L^2(I)}. 
\end{eqnarray}
Now the proof departs from \cite[Lemma 7]{S-2}, in that we choose $\alpha=\alpha_{\e,\ell}=2\ell-2 +2p_\e\frac{\ell-1}{k-1}$, to obtain 
\begin{eqnarray}\label{interp2-1}
\e^{\alpha_{\e,\ell}} \|u_\e^{(k)}\|^{2(1-\theta)}_{L^2(I)} \|u_\e^\prime\|^{2\theta}_{L^2(I)} 
\leq \e^{2p_\e-1} \Big(\e^{2k-1} \|u_\e^{(k)}\|^{2}_{L^2(I)}+\e^{1-2p_\e}\|u_\e^\prime\|^{2}_{L^2(I)}\Big);
\end{eqnarray}
that is, 
\begin{eqnarray}\label{interp2-2}
 \|u_\e^{(k)}\|^{2(1-\theta)}_{L^2(I)} \|u_\e^\prime\|^{2\theta}_{L^2(I)} 
\leq \e^{1-2\ell-2p_\e\frac{k-\ell}{k-1}} G_\e(u_\e;I)
\end{eqnarray}
and eventually, plugging this inequality in \eqref{interp1}, we obtain the desired estimate 
\eqref{interp2-3}.
\end{remark}
 
 \subsubsection{Estimate of sets below threshold without points of small derivatives}\label{estt-1}
For any fixed $N\geq 1$ and for any $\e>0$ we introduce the set $\mathcal I_\e(N)\subset \mathcal I_\e$ of the indices $j$ such that 
\begin{equation}\label{defIN}
\Big|\Big\{t\in I_j: |u_\e^{(\ell)}(t)|<\frac{1}{N\e^\ell} \ \hbox{\rm for all } \ell\in\{2,\dots, k-1\}\Big\}\Big|>0;
\end{equation} 
that is, thinking of $N$ as a large number, those are indices whose corresponding interval contains points $t$ where all the derivatives are small from order $2$ to $k-1$.

If \eqref{defIN} fails; that is, at every point the derivative of some order $\ell$ is ``large'', then we can give an estimate of the size of $I_j$. This can be done by using \eqref{interp2-3} for some $\ell$. More in general, we have the following result. 

\begin{lemma}\label{lowerlemma} 
Let $N\geq 1$ be fixed,  let $\e\in(0,1)$ and $I\subseteq (0,1)$ be an open interval satisfying 

{\rm (i)} $|u_\e^\prime(t)|<\frac{c_\e}{\e}$ for all $t\in I$; 

{\rm (ii)} the set $\{t\in I: |u_\e^{(\ell)}(t)|<\frac{1}{N\e^\ell} \ \hbox{\rm for all } \ell\in\{2,\dots, k-1\}\}$ has measure $0$. 

\noindent
Then we have $|I|\leq R_k(N) \e^{1+\frac{p_\e}{k-1}}$ for some positive constant $R_k(N)$ independent of $\e$ and $I$. In particular, the estimate holds for all $\e\in(0,1)$ and $I_j$ with $j\in \mathcal I_\e\setminus \mathcal I_\e(N)$. 
\end{lemma}

\begin{proof} The proof follows word for word that of \cite[Lemma 6]{S-2}, with the due changes necessary having modified the local interpolation inequality \eqref{interp2-3}. As a result, for all $\ell
\in\{2,\ldots,k-1\}$ we obtain an estimate
$$
|I|\leq R_k(N) \max\{\e^{1+2p_\e\frac{k-\ell}{k-1}}, \e^{\frac{2\ell-1}{2(\ell-1)}}\}=R_k(N) \e^{1+2p_\e\frac{k-\ell}{k-1}}
$$
for $\e$ small enough since $p_\e\to 0$. This implies the claim of the lemma.
\end{proof}

\begin{remark}\label{exponents}\rm
In the following the term $\e^{\frac{p_\e}{k-1}}$ will play the role that the term $\e^{\frac{1}{k-2}}$ has in \cite{S-2}, their relevant property is to be both infinitesimal as $\e\to 0$. For the first one this is ensured by claim (a) in Lemma \ref{MN-lemma}.
\end{remark}

If $j\in \mathcal I_\e(N)$, then we define 
\begin{eqnarray}\label{ab}\nonumber
&&a_j^{\e,N}\coloneq\inf\Big\{t\in I_j: |u_\e^{(\ell)}(t)|\le\frac{1}{N\e^\ell} \ \hbox{\rm for all } \ell\in\{2,\dots, k-1\}\Big\}\\
&&b_j^{\e,N}\coloneq\sup\Big\{t\in I_j: |u_\e^{(\ell)}(t)|\le\frac{1}{N\e^\ell} \ \hbox{\rm for all } \ell\in\{2,\dots, k-1\}\Big\}.
\end{eqnarray}
We set $$\mathcal A_\e(N)\coloneq \bigcup_{j\in \mathcal I_\e(N)} (a_j^{\e,N},b_j^{\e,N}). $$

Note that
\begin{equation}
|I_j\setminus (a_j^{\e,N},b_j^{\e,N})|\le 2R_k \e^{1+\frac{p_\e}{k-1}},
\end{equation}
since $I_j\setminus (a_j^{\e,N},b_j^{\e,N})$ is composed of two boundary intervals satisfying the hypotheses of Lem\-ma~{\rm\ref{lowerlemma}} by definition. 

\bigskip
Lemma \ref{lowerlemma} also allows us to provide an estimate for the measure of $I\cap  \mathcal A_\e$ for small intervals $I$ without points below threshold and with small derivatives.

\begin{lemma}\label{osck-lemma}  
If $I\subseteq (0,1)$ is an open interval such that $|I|\le\e$ and $I\cap \mathcal A_\e(N)=\emptyset$, 
then, $\#\{j\in\mathcal I_\e: I_j\cap I\neq \emptyset\}\leq k$. In particular, by Lemma {\rm\ref{lowerlemma}} we have $|I\cap  \mathcal A_\e|\le kR_k(N) \e^{1+\frac{p_\e}{k-1}}$.
\end{lemma} 

\begin{proof} The proof is exactly as in \cite[Lemma 8]{S-2}, since we do not use any property of the energies $F_\e$. 
We briefly illustrate the argument by contradiction, referring to the proof of \cite[Lemma 8]{S-2} for finer details. Suppose that there exist $\e_n\to0$ and $I^n$ intervals with  $|I^n|\le\e_n$ and $I^n\cap \mathcal A_{\e_n}(N)=\emptyset$, but $\#\{j\in\mathcal I_{\e_n}: I_j\cap I^n\neq \emptyset\}\ge k+1$. Note that at the boundary of each interval in the complement of $A_{\e_n}$ the values of $u_{\e_n}$ coincide; hence, by Rolle's theorem there exists a point in that interval such that $u'_{\e_n}=0$. We deduce that there exist at least $k$ such points. Using the same argument, we prove that there exist $k-1$ points where $u''_{\e_n}=0$, and, proceeding iteratively, in particular deduce that for each $\ell\in\{1,\ldots,k-1\}$ there exists at least one point such that $u^{(\ell)}_{\e_n}=0$. Up to translations, we can suppose that $I^n=(0,\lambda_n)$ for some $\lambda_n\le \e_n$. We then consider the functions $v_n(t)\coloneq u_{\e_n}(\lambda_n t)$, which satisfy $v^{(\ell)}_n=0$ at some point in $(0,1)$, and their derivatives are equi-bounded in $H^{k-1}$ by the boundedness of the energy and Poincar\'e's inequality.  Using the $1/2$-H\"older continuity of all $v_n^{(\ell)}$ this is in contradiction with supposing that $I^n\cap \mathcal A_{\e_n}(N)=\emptyset$.
\end{proof} 

%

The following result states that intervals without points both below threshold and with small derivatives have a small portion below threshold. 

\begin{proposition}[Asymptotics for the relative measure of intervals below threshold]\label{stimatausigma1} 
Let $N\geq 1$ be a fixed integer and  $\e\in(0,1)$. 
If  
$\{I(\e)\}$ is a family of intervals such that
$|I(\e)|>\!>\e^{1+\frac{p}{k-1}}$
and $I(\e)\cap \mathcal A_\e(N) 
=\emptyset$,  
then 
\begin{equation}\label{stimaprop}
\lim_{\e\to 0}\frac{|I(\e) \cap 
\mathcal A_\e |}{|I(\e)|}
=0. 
\end{equation}
\end{proposition}

\begin{proof}
The  proof is the same as that of Proposition 9,  Remark 10 and Proposition 11 in \cite{S-2}, up to substituting the exponent $\frac{4k-3}{4k-4}$  with ${1+\frac{p_\e}{k-1}}$.
\end{proof}

This result will be used to estimate the contribution of $G_\e(u_\e, I(\e))$ when $I(\e)$ are maximal intervals in $(0,1)\setminus A_\e(N)$, of the form $I(\e)=[b_j^{\e,N}, a_{j'}^{\e,N}]$, and \eqref{stimaprop} holds. In this case, if $z_\e\coloneq u_\e (a_{j'}^{\e,N})-u_\e (b_{j}^{\e,N})$, we have 
\begin{eqnarray}\label{pre-fa}\nonumber
G_\e(u_\e, I(\e))&\ge& \frac1\e|I(\e)\setminus \mathcal A_\e |+ \e^{2k-1} \int_{I(\e)} |u_\e^{(k)}|^2\dx\\
\nonumber&\ge&\frac{|I(\e)\setminus \mathcal A_\e |}{|I(\e)|}\Big( \frac1\e|I(\e)|+ \e^{2k-1} \int_{I(\e)} |u_\e^{(k)}|^2\dx\Big)\\ \nonumber
&\ge&\frac{|I(\e)\setminus \mathcal A_\e |}{|I(\e)|}\min_{T,v}\biggl\{T +\int_0^T |v^{(k)}|^2dt:  v(T)-v(0)= z_\e,\   |v'(0)|, |v'(T)|\le c_\e,\\
&&\hskip 2cm |v^{(\ell)}(0)|, |v^{(\ell)}(T)|\le \frac1N\hbox{ for all }\ell\in\{2,\ldots, k-1\}\Big\},
\end{eqnarray}
where we have scaled the function $u_\e$ and taken into account the boundary conditions for $\ell\in\{1,\ldots, k-1\}$ due to the definition of $A_\e(N)$.
We can also further bound the minimum in the last line  by
\begin{eqnarray*}
&&\min_{T,v}\biggl\{T +\int_0^T |v^{(k)}|^2dt:  v(T)-v(0)= z_\e,\\
&&\hskip 2cm |v^{(\ell)}(0)|, |v^{(\ell)}(T)|\le \frac1N\hbox{ for all }\ell\in\{1,\ldots, k-1\}\Big\},
\end{eqnarray*}
upon assuming that $c_\e\le \frac1N$. 
We will show that this minimum provides a good lower bound for vales of $z_\e$ uniformly far from $0$, while a finer argument must be used for small $z_\e$. This additional analysis is necessary since energy densities depending on the jump size must have an infinite derivative at $0$.

\subsubsection{Fine analysis of small sets below threshold}\label{estt-2}
Still following \cite{S-2}, we introduce new families of intervals in $\mathcal I_\e$ in order to characterize the limit energy density for small values of the jump size. With fixed $r\in(0,1)$ we set
\begin{equation}\label{irr}
\mathcal I^r_\e=\Big\{ j\in \mathcal I_\e : \int_{I_j} |u'_\e|^2 dt<r\Big(\frac{c_\e}\e\Big)^2\,|I_j|\Big\};
\end{equation}
that is, intervals below threshold and such that 
$$
\e^{1-2p_\e} \int_{I_j} |u'_\e|^2 dt< r\frac{|I_j|}\e.
$$
Note that the left-hand side of this equation is the part of $G_\e(u_\e, I_j)$ depending on $u'_\e$ in \eqref{defGI}. The value of $r$ in the rest of the paper will be fixed such that it satisfies the claim of the following lemma, which shows that we have a lower bound for the length of each interval in $\mathcal I^r_\e$. Note that we may suppose that $r$ is such that the number of intervals in $\mathcal I^r_\e$ is finite.

\begin{lemma}
\label{upperlemma}  
There exist a threshold $r=r_k\in (0,1)$ 
and a constant $\widetilde C=\widetilde C_k>0$ 
such that for all $\e\in(0,1)$ and for all intervals $I$ satisfying 

{\rm (i)} $|u_\e^\prime|<\frac{c_\e}{\e}$ in $I$ and $|u_\e^\prime|=\frac{c_\e}{\e}$ at at least one of the endpoints,

{\rm (ii)} $\int_I (u^\prime_\e)^2\, dt<r (\frac{c_\e}\e )^2|I|$,

\noindent
we have $|I|\geq \widetilde C \e^{1+\frac{2p_\e}{k-1}}$.
In particular, the claim holds for all $\e\in(0,1)$ and $I_j$ with $j\in  \mathcal I_\e^{r}$. 
\end{lemma}

\begin{proof}  The proof is the same as that of \cite[Lemma 12]{S-2}, again up to substituting the exponent $\frac{4k-3}{4k-4}$ therein with ${1+\frac{p_\e}{k-1}}$.
\end{proof}

For any $N\geq 1$, we define $\mathcal I_\e^\ast(N)$ as the set of indices $j\in \mathcal I^r_\e$ such that 
\begin{equation}\label{definast} 
\Big|\Big\{t\in I_j: |u_\e^{(\ell)}(t)|<\frac{1}{N\e^\ell} \ \hbox{\rm for all } \ell\in\{1,\dots, \ell(k)\}\Big\}\Big|>0, 
\end{equation} 
where $\ell(k)=\max\{ \ell\in\mathbb N: 2\ell<k+1\}$.
Note that, since $|u_\e'|<\frac{c_\e}{\e}$ in each $I_j$, 
then the inequality for $\ell=1$ is true for all $t\in I_j$ as soon as $c_\e<\frac{1}{N}$. 
In this case, $\mathcal I_\e^\ast(N)$ contains $\mathcal I_\e(N)$.

The following lemma states that the union of all intervals in $\mathcal I_\e^{r}$ whose points do not have small derivatives up to order $\ell(k)$ as in \eqref{definast} has a small measure. We use the notation
$$
\mathcal A_\e^r\coloneq\bigcup_{j\in \mathcal I_\e^r}I_j,\qquad \mathcal A_\e^{\ast}(N)\coloneq\bigcup_{j\in \mathcal I_\e^\ast(N)}I_j.
$$

\begin{lemma}\label{lemmar}  
Let $N\geq 1$ be fixed  and let $r$ be fixed as above. 
Then there exists a threshold $\e(N)\in (0,1)$  
such that $|\mathcal A_\e^r\setminus \mathcal A_\e^{\ast}(N)|\leq R_k k^2 S N^2\e^{1+p_\e}$
for all $\e\in (0,\e(N))$. 
\end{lemma}

\begin{proof}  The proof is the same as that of \cite[Lemma 13]{S-2}, again up to substituting the exponent $\frac{4k-3}{4k-4}$ therein with ${1+\frac{p_\e}{k-1}}$.
\end{proof}

\begin{remark}\label{eqrbordo}\rm
In analogy to the definition of $(a_j^{\e,N},b_j^{\e,N})$ in \eqref{ab} for all $j\in\mathcal I_\e^\ast(N)$, we let $(a_{*,j}^{\e,N},b_{*,j}^{\e,N})$ denote the maximal subinterval  of $I_j$ such that $|u_\e^{(\ell)}|\leq\frac{1}{N\e^\ell}$ for all $\ell\in\{1,\dots, \ell(k)\}$ at the endpoints. Then we also have that the total sum of the length of the intervals obtained as one of the two intervals in $I_j\setminus (a_{*,j}^{\e,N},b_{*,j}^{\e,N})$ and satisfying the inequality in \eqref{irr} is bounded by $R_k k^2 S N^2\e^{1+p_\e}$ (as in \cite[Remark 14]{S-2}).
\end{remark}

This result will be used to estimate the contribution of $G_\e(u_\e, I(\e))$ when $I(\e)$ are maximal intervals of the form $I(\e)=[b_{*,j}^{\e,N}, a_{*,j'}^{\e,N}]$. In this case, if $z_\e\coloneq u_\e (a_{*,j'}^{\e,N})-u_\e (b_{*,j}^{\e,N})$,
setting
\begin{eqnarray}\label{Ire}
I_r^*(\e)= (I(\e)\cap ((I_j\setminus (a_{*,j}^{\e,N},b_{*,j}^{\e,N}))\cup I_{j'}\setminus (a_{*,j'}^{\e,N},b_{*,j'}^{\e,N}))
\cup (I(\e)\setminus \mathcal A_\e^{\ast}(N)),
\end{eqnarray}
so that $I(\e)\setminus I_r^*(\e)$ is a union of intervals $I$ such that $
\e^{1-2p_\e} \int_{I} |u'_\e|^2 dt\ge r\frac{|I|}\e$ holds, we have
\begin{eqnarray}\label{stimabi}\nonumber
&&\hskip-2cm G_\e(u_\e, I(\e))\ =\ \e^{1-2p}\int_{I(\e)\cap I_r^*(\e)} |u'_\e|^2dt \nonumber
+\int_{I(\e)\setminus I_r^*(\e)} |u'_\e|^2dt+ \e^{2k-1} \int_{I(\e)} |u_\e^{(k)}|^2\dx\\ \nonumber
&\ge& r\frac{|I(\e)\setminus I_r^*(\e)|}{\e}+ \e^{2k-1} \int_{I(\e)} |u_\e^{(k)}|^2\dx\\ \nonumber
&=& r\frac{|I(\e)|}{\e}+ \e^{2k-1} \int_{I(\e)} |u_\e^{(k)}|^2\dx
-r\frac{|I_r^*(\e)|}{\e}\\ \nonumber
&\ge&r\min_{T,v}\biggl\{T +\int_0^T |v^{(k)}|^2dt:  v(T)-v(0)= z_\e, |v'(0)|, |v'(T)|\le c_\e \\ &&\hskip 2cm |v^{(\ell)}(0)|, |v^{(\ell)}(T)|\le \frac1N\hbox{ for all }\ell\in\{2,\ldots, \ell(k)\}\Big\}
-r\frac{|I_r^*(\e)|}{\e},
\end{eqnarray}
where we have scaled the function $u_\e$ and taken into account the boundary conditions for $\ell\in\{2,\ldots, \ell(k)\}$. Note that the sum of the last terms over all such intervals $I(\e)$ is negligible, since it is at most of order $\e^p$  by Lemma \ref{lemmar} and Remark \ref{eqrbordo}.

\subsection{Energy densities}
In this section we have gathered some definitions of energy densities and proved some of their properties that are useful for the computation of the $\Gamma$-limit. These energy densities are now essentially the same as in \cite{S-2}, with $c_\e$ playing the role of an arbitrary sequence converging to $0$, so that most of the results are already contained there. We include these sections by completeness and because we use slightly different definitions.

\subsubsection{A lower bound with (large) linear growth at $0$.}
We now provide some lower estimates for the function 
\begin{eqnarray}\label{phieN}\nonumber
&&\phi_{\e,N}(z):=
\inf_{T,v}\biggl\{T +\int_0^T |v^{(k)}|^2dt:  v(T)-v(0)= z, |v'(0)|, |v'(T)|\le c_\e, \\
&&\hskip 4cm |v^{(\ell)}(0)|, |v^{(\ell)}(T)|\le \frac1N\hbox{ for all }\ell\in\{2,\ldots, \ell(k)\}\Big\}.
\end{eqnarray}
We first note that the minimum is achieved for $T$ strictly larger than a constant, if $|z|\ge\theta>0$.
Upon taking $\e$ small enough so that $c_\e\le \frac1N$, this claim is asserted by the following proposition.

\begin{proposition}\label{lobaTe}
Let $\theta>0$, $T_\e\to 0$ and $z_\e\ge \theta$ as $\e\to 0$. Then 
\begin{eqnarray*}\label{minTe}
&&\lim_{\e\to 0}\min_{v}\biggl\{\int_0^{T_\e} |v^{(k)}|^2dt:  v(T)-v(0)= z_\e,\\
&&\hskip 2cm |v^{(\ell)}(0)|, |v^{(\ell)}(T)|\le \frac1N\hbox{ for all }\ell\in\{1,\ldots, \ell(k)\}\}\Big\}=+\infty
\end{eqnarray*}
\end{proposition}

\begin{proof}
The proof is achieved by contradiction: suppose otherwise, then if $\{v_\e\}$ is a sequence of minimizers, the functions $w_\e(t)\coloneq v_\e(T_\e t)$ belong to $H^{(k)}(0,1)$ and $\int_0^{1} |w_\e^{(k)}|^2dt\to 0$. Using iteratively Poincar\'e-Wirtinger's inequality, we prove that (after passing to a subsequence) their limit is a polynomial $P$ of degree at most $k-1$ with homogeneous conditions on $P^{(\ell)}$ on $0$ and $1$ for $\ell\in\{1,\ldots, \ell(k)\}$. This implies that $P$ is a constant,  in contradiction with the boundary condition $|P(1)-P(0)|\ge \theta$. For details we refer to \cite[Remark 4]{S-2}.
\end{proof}

We now estimate the value of $\phi_{\e,N}(z)$ for
$z$ close to $0$ by an affine function whose coefficient diverges as $\e\to 0$,
and involving the quantity
\begin{eqnarray}\label{mN}\nonumber
&&m(N):=
\inf_{T,v}\biggl\{T +\int_0^T |v^{(k)}|^2dt:  v(T)-v(0)= 1,  \\
&&\hskip 3cm |v^{(\ell)}(0)|, |v^{(\ell)}(T)|\le \frac1N\hbox{ for all }\ell\in\{1,\ldots, \ell(k)\}\Big\},
\end{eqnarray}
which is proved to be strictly positive in \cite[Remark 4]{S-2}. 
\begin{proposition} {Let $k\ge 3$.} For $|z|\le c_\e^3$ we have $\phi_{\e,N}(z)\ge  \frac{m(N)}{c_\e N}|z|$.
\end{proposition}

\begin{proof} Let $\e,N,z$ be fixed and let $v,T$ be admissible test items for \eqref{phieN}. Without loss of generality $z>0$. We define
$\beta\coloneq \frac{z}{c_\e N}$, $T'\coloneq\frac{T}{\beta}$, and 
$w(s)\coloneq \frac{1}z v(\beta s)$.

We have 
\begin{equation}\label{dis-1}
\int_0^T|v^{(k)}|^2dt= \frac{z^2}{\beta^{2k-1}}\int_0^{T'}  |w^{(k)}|^2ds\ge \beta\int_0^{T'}  |w^{(k)}|^2ds
\end{equation}
since 
$z^{k-2}\le c_\e^{3(k-2)}\le c_\e^{k}\le N^{k} c_\e^{k}$.
Hence, we obtain
\begin{equation}\label{dis-2}
T +\int_0^T |v^{(k)}|^2dt\ge \beta\Big(T'+\int_0^{T'}  |w^{(k)}|^2ds\Big)
\end{equation}
Moreover, we have $w(T')-w(0)= 1$,  $|w'(0)|, |w'(T')|\le \frac1N$, and 
$$
|w^{(\ell)}(0)|, |w^{(\ell)}(T')|\le \frac{\beta^\ell}{Nz}=\frac{z^{\ell-1}}{c^\ell_\e N^{\ell+1}}, 
$$
which is less than $\frac1N$ since $z^{\ell-1}\le c_\e^{3(\ell-1)}\le c_\e^\ell$,
so that $(T',w)$ is an admissible test pair for \eqref{mN}, and we obtain the claim.
\end{proof}

\begin{proposition} Let $\theta\in (0,1)$ and $N$ be fixed. Then there exits $\e(\theta,N)$ such that we have $\phi_{\e,N}(z)\ge  m(N)\theta^{\frac1k-\frac34} |z|^{\frac34}$ for all $\e\le\e(\theta,N)$ and for all $z$ with $|z|\in [c_\e^3,\theta]$.
\end{proposition}

\begin{proof} The proof is analogous to that of the previous proposition, taking $\beta=\theta^{\frac1k-\frac34} |z|^{\frac34}$ as in the proof of \cite[Proposition 16]{S-2}. 
With this choice, \eqref{dis-1} holds since the inequality $\frac{z^2}{\beta^{2k}}\ge 1$ is guaranteed by condition $|z|\le \theta$, and we obtain  \eqref{dis-2}.

 Regarding the boundary values, we have $|w^{(\ell)}(0)|, |w^{(\ell)}(S)|\le\frac1N$, while the conditions on the first derivative give
$|w'(0)|,|w'(S)|\le\theta^{\frac1k-\frac34} |z|^{\frac34}c_\e$. This value is less than $\frac1N$  for $|z|\ge c_\e^3$ if $c_\e^{{13}/4}\le \frac1N\theta^{\frac34-\frac1k}$, so for these values of $\e$ we can use $(S,w)$ as a test pair for \eqref{mN}, and we obtain the claim.
 The equality $c_\e^{{13}/4}= \frac1N\theta^{\frac34-\frac1k}$ is the implicit definition of $\e(\theta,N)$.
\end{proof} 

\begin{proposition} Let $\theta\in (0,1)$ and $N$ be fixed. Then, upon possibly redefining $\e(\theta,N)$, we have $\phi_{\e,N}(z)\ge  m({\theta N}) |z|^{1/k}$ for all $\e\le\e(\theta,N)$ and for all $z$ with $|z|\ge \theta$, where $m({\theta N})$ is defined as in \eqref{mN} with $\theta N$ is the place of $N$. 
\end{proposition}

\begin{proof}
The proof is analogous to that of the previous propositions, taking $\beta=|z|^{1/k}$. 
\end{proof}

From the two previous propositions we have the following estimate.

\begin{corollary}\label{psinn} Let $\theta\in (0,1)$ and $N$ be fixed. Then we have $\phi_{\e,N}(z)\ge\psi_{\theta,N}(z)$ for all $\e\le\e(\theta,N)$ and for all $z$, where  $$\psi_{\theta,N}(z):= \min\{m(N)\theta^{\frac1k-1} |z|, m({\theta N})|z|^{1/k}\}.$$ 
\end{corollary}

\subsubsection{Asymptotic optimal lower bound}\label{oplo}
We now describe the asymptotic behaviour of the functions 
\begin{eqnarray*}\label{minT}
&&\phi_N(z)=\min_{T,v}\biggl\{T +\int_0^T |v^{(k)}|^2dt:  v(T)-v(0)= z,\\
&&\hskip 2cm |v^{(\ell)}(0)|, |v^{(\ell)}(T)|\le \frac1N\hbox{ for all }\ell\in\{1,\ldots, k-1\}\Big\}.
\end{eqnarray*}

\begin{proposition}
The functions $\phi_N$ converge as $N\to+\infty$ increasingly to the function $m_k|z|^{1/k}$, 
where
\begin{eqnarray*}
&&m_k\coloneq\min_{T,v}\biggl\{T +\int_0^T |v^{(k)}|^2dt:  v(T)-v(0)= 1,\\
&&\hskip 2cm |v^{(\ell)}(0)|, |v^{(\ell)}(T)|=0\hbox{ for all }\ell\in\{1,\ldots, k-1\}\Big\}.
\end{eqnarray*}
Furthermore, the convergence is uniform on compact subsets of $\mathbb R\setminus\{0\}$.
\end{proposition}

\begin{proof} Since 
\begin{eqnarray*}
&&m_k|z|^{1/k}=\min_{T,v}\biggl\{T +\int_0^T |v^{(k)}|^2dt:  v(T)-v(0)= z,\\
&&\hskip 2cm |v^{(\ell)}(0)|, |v^{(\ell)}(T)|=0\hbox{ for all }\ell\in\{1,\ldots, k-1\}\Big\},
\end{eqnarray*}
which is the pointwise increasing limit of $\phi_N(z)$, it is enough to prove that we may apply Ascoli-Arzel\`a's Theorem. Since the functions $\phi_N$ are even and monotone for $z>0$, they are equibounded on bounded sets. Hence, we only need to show that $\phi_N$ are equicontinuous.
To this end, it suffices to consider $\phi_N$ on a subinterval $[z_1,z_2]\subset (0,+\infty)$. 
Let $z,z'\in [z_1,z_2]$, and we may suppose that $z'>z$. Let $(T,v)$ be a test pair for $\phi_N(z)$. We set $S\coloneq\frac{z'}{z} T$ and $w(s) \coloneq \frac{z'}z v\big(\frac{z}{z'} s\big)$, so that $w(S)-w(0)= z'$, $|w^{(\ell)}(0)|=|\frac{z}{z'}|^{\ell-1} |v^{(\ell)}(0)|\le  |v^{(\ell)}(0)|$ and similarly for $|w^{(\ell)}(T)|$. Using $(S,w)$ as a test pair for $\phi_N(z')$, we have
\begin{eqnarray*}
\phi_N(z')\le S+\int_0^S |w^{(k)}|^2ds = \frac{z'}{z} T+ \Big(\frac{z}{z'}\Big)^{2k-3}\int_0^T |v^{(k)}|^2dt
\le \frac{z'}{z} T+\int_0^T |v^{(k)}|^2dt,\end{eqnarray*}
which shows that $\phi_N(z')\le  \frac{z'}{z}\phi_N(z)$, and hence  
$$|\phi_N(z')-\phi_N(z)|=\phi_N(z')-\phi_N(z)\le  \Big(\frac{z'}{z}-1\Big)\phi_N(z)\le m_k\frac{|z_2|^{1/k}}{|z_1|}|z'-z|,$$
proving the equi-Lipschitz continuity of $\phi_N$ on $[z_1,z_2]$. 
\end{proof}

\begin{remark}\rm
If $(\overline T,\overline v)$ denotes the minimal pair for $m_k$, then the minimal pair $(T_z,v_z)$ for the problem with the boundary condition $v(T)-v(0)=z>0$ is given by
$$
T_z\coloneq z^{1/k} \overline T,\qquad v_z(t)\coloneq z\,\overline v( {t}{z^{-\frac1k}}).
$$
\end{remark}

\begin{remark}\label{lobaTe-2}\rm
Note that Proposition \ref{lobaTe} holds a fortiori also for the problems involving boundary conditions on all derivatives up to $k$.
\end{remark}

\section{One-dimensional $\Gamma$-limit and compactness}\label{ga-sec}
We can now proceed with the proof of the convergence theorem in dimension one. The analysis in the previous sections following \cite{S-2} for the jump part of the limit energy will be combined with arguments in \cite{M-N} that allow one to treat separately the bulk part of the limit energy.

\begin{theorem}[compactness]\label{compactness}
Let $\{u_\e\}$ be a family of functions in $H^k(0,1)$ such that $F_\e(u_\e)\le S<+\infty$. 
Then, up to extraction of a subsequence and additions of constants, $\{u_\e\}$ converges in measure to a function $u\in SBV(0,1)$.
\end{theorem}

\begin{proof}
The scheme of the proof is as follows. We consider $N\in\mathbb N$ fixed. Note that in the following argument $N$ will not play any role, and we can take $N=1$. 
With fixed $\theta\in(0,1)$, we define the functional 
\begin{equation}
\Psi_{\theta,N}(v)\coloneq\int_0^1 |v'|^2dt
+ \sum_{t\in S(v)}\psi_{\theta,N}(v(t+)-v(t-))
\end{equation}
for $v\in SBV(0,1)$, where $\psi_{\theta,N}$ is as in Corollary \ref{psinn}.
We construct functions $v_\e^\theta$, such that $u_\e-v_\e^\theta$ tends to $0$ in measure and 
$\Psi_{\theta,N}(v_\e^\theta)$ is equibounded independently of $\theta$. From the compactness theorem in $BV$, we deduce that $\{v_\e^\theta\}$ converges weakly in $BV$ to a function $v^\theta$. We then prove that $\{v^\theta\}$ tends in $L^1(0,1)$ as $\theta\to 0$  to a function $v_0\in SBV$.

\smallskip For the sake of notation, we simply write $v_\e$ in place of $v_\e^\theta$.
The construction of $v_\e$ combines arguments from \cite{S-2} (for the jump part) and \cite{M-N} (for the bulk part). It consists of two steps:

\smallskip
(i) (bulk part) let 
$$
D_\e:= \Big\{t\in (0,1): |u_\e'|^2>\tfrac{\sqrt{c_\e}}{\e\sqrt{|\log\e|}}\Big\},
$$
and let $D^*_\e$ be the intersection of $D_\e$ with the
complement of the (finite) union of the intervals $[b_{*,j}^{\e,N}, a_{*,j'}^{\e,N}]$ as introduced in Section \ref{estt-2}. We  then consider the functions $\widetilde u_\e$ defined 
by $\widetilde u_\e(0)\coloneq u_\e(0)$ and 
$\widetilde u'_\e\coloneq  u'_\e \chi_{(0,1)\setminus D^*_\e}$; that is, we modify $u_\e$ by making it constant on $D^*_\e$. Note that  $\int_0^1\frac1{\e|\log\e|}\log (1+\e|\log\e||u'_\e|^2)dt\le S<+\infty$, then $\lim\limits_{\e\to 0}\frac{c_\e}\e|D_\e|=0$, which implies that the modification is asymptotically negligible in $L^1(0,1)$.

\smallskip
(ii) (jump part) intervals $[b_{*,j}^{\e,N}, a_{*,j'}^{\e,N}]$, as defined in Section \ref{estt-2}, correspond to jumps for $v_\e$. This function is defined as the constant $\widetilde u_\e(b_{*,j}^{\e,N})$ on the interval $[b_{*,j}^{\e,N}, a_{*,j'}^{\e,N})$, so that it has a jump of size $z_\e= u_\e(a_{*,j'}^{\e,N})-u_\e(b_{*,j}^{\e,N})$ at the point $a_{*,j'}^{\e,N}$. In the complement of the union of such intervals, we set $v_\e\coloneq\widetilde u_\e$

\bigskip
By \eqref{stimabi}, we have 
$$
\sum_{t\in S(v_\e)}\psi_{\theta,N} (v_\e(t+)-v_\e(t-))
\le \frac1r\sum_{j} G_\e(u_\e, (b_{*,j}^{\e,N}, a_{*,j'}^{\e,N}))
+C\e^p\le \frac1r F_\e(u_\e)+o(1),
$$
taking into account claim (b) of Lemma \ref{MN-lemma} and the definition of $r$ in Lemma \ref{upperlemma}. 
The error term comes from the sum of the rests in \eqref{stimabi}, taking into account
Lemma \ref{lemmar} and Remark \eqref{eqrbordo}. The constant $C$ is given by Lemma \ref{lemmar}, and depends on $N$ but not on $\theta$.

As for the bulk contribution we note that, with fixed $\eta\in(0,1)$ there exists $\e(\eta)$ such that for  $\e\le \e(\eta)$ we have 
 $$
(1-\eta)|z|^2\le\frac1{\e|\log\e|}\log (1+\e|\log\e||z|^2)
 $$ 
for all $|z|^2\le \tfrac{\sqrt{c_\e}}{\e\sqrt{|\log\e|}}$.
Taking $\eta=\frac12$ this implies that
$$
\int_0^1|v'_\e|^2 dt \le 2\int_0^1\frac1{\e|\log\e|}\log (1+\e|\log\e||v'_\e|^2)dt
\le 2 F_\e(u_\e).
$$
From these two estimates we obtain that $\Psi_{\theta, N} (v_\e)\le (2+\frac1r)S +o(1)$,
We note that, since $\psi_{\theta, N}(z)$ diverges as $|z|$ tends to $+\infty$, this estimate implies that the maximal jump size of $v_\e$ is equibounded, and there exists a positive constant $\kappa$ such that $\psi_{\theta, N}(v_\e(t+)-v_\e(t-))\ge \kappa |v_\e(t+)-v_\e(t-)|$. This implies that, up to the addition of constants, $v_\e$ is precompact in $BV(0,1)$-weak$^*$, so that we can assume that $v_\e\to v$. 
If $\overline\Psi_{\theta, N}$ denotes the lower-semicontinuous envelope of $\Psi_{\theta, N}$ with respect to the weak$^*$ topology in $BV(0,1)$, then we have
$$
 \Big(2+\frac1r\Big)S\ge \liminf_{\e\to 0} \Psi_{\theta, N}(v_\e)\ge \overline\Psi_{\theta, N}(v).
$$
Since $\{v_\e-u_\e\}$ tends to $0$ in measure, this implies that $\{u_\e\}$ tends to $v$ in measure.

We now show that, by the arbitrariness of $\theta\in(0,1)$, the limit is in $SBV(0,1)$.
By \cite{BBB} we have 
$$
\Psi_{\theta, N}(v)=\int_{(0,1)} f_{\theta, N} (v')dt+ m(N)\theta^{\frac1k-1}\|D_Cv\|(0,1)
+\sum_{t\in S(v)} \psi_{\theta, N}(v(t+)-v(t-)),
$$
where $D_Cv$ denotes the Cantor part of the derivative of $v$, and 
$$
f_{\theta, N}(z)=(|z|^2\wedge m(N)\theta^{\frac1k-1}|z|)^{**}.
$$
This implies that
$$
\|D_Cv\|(0,1)\le  \frac{S}{m(N)}\Big(2+\frac1r\Big)\theta^{1-\frac1k}.
$$
If $\|D_Cv\|(0,1)>0$ this would give a contradiction for $\theta$ small enough.
\end{proof}

\begin{remark}[A sub-optimal lower bound]\label{som-1}\rm
From the proof of the previous theorem, localized in any open subinterval $I$ of $(0,1)$ and optimized in the estimates involving $\eta$ and $r$, we have
$$
\Big(\Gamma\hbox{-}\liminf_{\e\to 0} F_\e\Big)(v, I)\ge \int_{I} f^r_{\theta, N} (v')dt+ r m(N)\theta^{\frac1k-1}\|D_Cv\|(I)
+ r\sum_{t\in S(v)\cap I} \psi_{\theta, N}(v(t+)-v(t-)),
$$ 
where $f^r_{\theta, N}(z)=(|z|^2\wedge r\, m(N)\theta^{\frac1k-1}|z|)^{**}$.
Optimizing locally such estimates, we deduce that 
\begin{eqnarray*}
\Big(\Gamma\hbox{-}\liminf_{\e\to 0} F_\e\Big)(v, I)&\ge& \int_{I} \sup_{\theta\in(0,1), N}f_{\theta, N} (v')dt+r\sum_{t\in S(v)\cap I} \sup_{\theta\in(0,1), N}\psi_{\theta, N}(v(t+)-v(t-))\\
&=& \int_{I} |v'|^2dt+ r\sum_{t\in S(v)\cap I} m_k |v(t+)-v(t-)|^{1/k}
\end{eqnarray*}
(this is formalized in the so-called sup-of-measure lemma, see e.g.~\cite{Handbook}).
This estimate is not optimal by the presence of $r$, which is strictly less than $1$, but already shows that the domain of the $\Gamma$-limit is (contained in) the set of $SBV(0,1)$ functions whose approximate derivative is in $L^2(0,1)$ and such that $\sum\limits_{t\in S(v)\cap I}  |v(t+)-v(t-)|^{1/k}<+\infty$. \end{remark}

A finer use 
of the energy densities in Section \ref{oplo} will be necessary to prove the sharp lower bound in the following proposition.

\begin{proposition}[lower bound]
Let $\{u_\e\}$ be a family of functions in $H^k(0,1)$ such that $F_\e(u_\e)\le S<+\infty$, and let $u\in SBV(0,1)$ be such that $u_\e\to u$ in measure. Then we have
\begin{equation}
\liminf_{\e\to 0} F_\e(u_\e)\ge \int_{(0,1)}|u'|^2dt +m_k\sum_{t\in S(u)} |u(t+)-u(t-)|^{1/k}.
\end{equation} 
\end{proposition}

\begin{proof} We fix $N\in\mathbb N$. We refine the construction of functions $v_\e$ in the proof of Theorem \ref{compactness} using, where convenient, the intervals $[b_{j}^{\e,N}, a_{j'}^{\e,N}]$ as defined in Section \ref{estt-1} instead of the intervals $[b_{*,j}^{\e,N}, a_{*,j'}^{\e,N}]$ as defined in Section \ref{estt-2}. This will yield functions $v_\e$ with a jump in $a_{j'}^{\e,N}$ of size $z_\e$, whose contribution is estimated by $\phi_N(z_\e)$. Since $\phi_N(z)\to m_k |z|^{1/k}$ on compact subsets of $\mathbb R\setminus \{0\}$ as $N\to +\infty$, we again introduce a threshold $\theta\in(0,1)$ and consider different intervals $[b_{j}^{\e,N}, a_{j'}^{\e,N}]$ with $|u_\e(a_{j'}^{\e,N})-u_\e(b_{j}^{\e,N}) 
|\le\theta$. We then consider the intervals of the form $[b_{*,j_*}^{\e,N}, a_{*,j_*'}^{\e,N}]$ contained in $[b_{j}^{\e,N}, a_{j'}^{\e,N}]$, where we use the construction in the proof of Theorem \ref{compactness}. 

\bigskip
Since the superposition of the two constructions is not easy to follow, we introduce some notation similar to that of \cite{S-2}. For every $\e>0$ we denote by $\tau_n=\tau^\e_n$ ,$\sigma_n=\sigma^\e_n$, for $n\in\{1,\ldots, K_\e\}$, the endpoints of the disjoint intervals whose union is (the closure of) $(0,1)\setminus \bigcup_{j\in \mathcal I_\e(N)}\mathcal (a_{j}^{\e,N},b_{j}^{\e,N})$; 
that is 
$$
(0,1)\cap\bigcup_{n=1}^{K_\e}[\tau^\e_n,\sigma^\e_n]=(0,1)\setminus \bigcup_{j\in \mathcal I_\e(N)}\mathcal (a_{j}^{\e,N},b_{j}^{\e,N}).
$$
We parameterize the intervals so that they are increasing with $n$; that is, $\sigma^\e_n<\tau^\e_{n+1}$.
Since $\mathcal I_\e(N)$ is a finite set, then either $\tau^\e_n=0$ or $\tau^\e_n=b_{j}^{\e,N}$ for some $j\in  \mathcal I_\e(N)$, and correspondingly $\sigma^\e_n=\min\{a_{j'}^{\e,N}\in (b_{j}^{\e,N},1)\}$, with the convention that $\sigma^\e_n=1$ if there is no such $a_{j'}^{\e,N}$.
Note that some $[\tau^\e_n,\sigma^\e_n]$ may contain intervals $I^\e_j$ with $j\in \mathcal I_\e\setminus \mathcal I_\e(N)$.

\begin{figure}[h!]
\centerline{\includegraphics[width=0.7\textwidth]{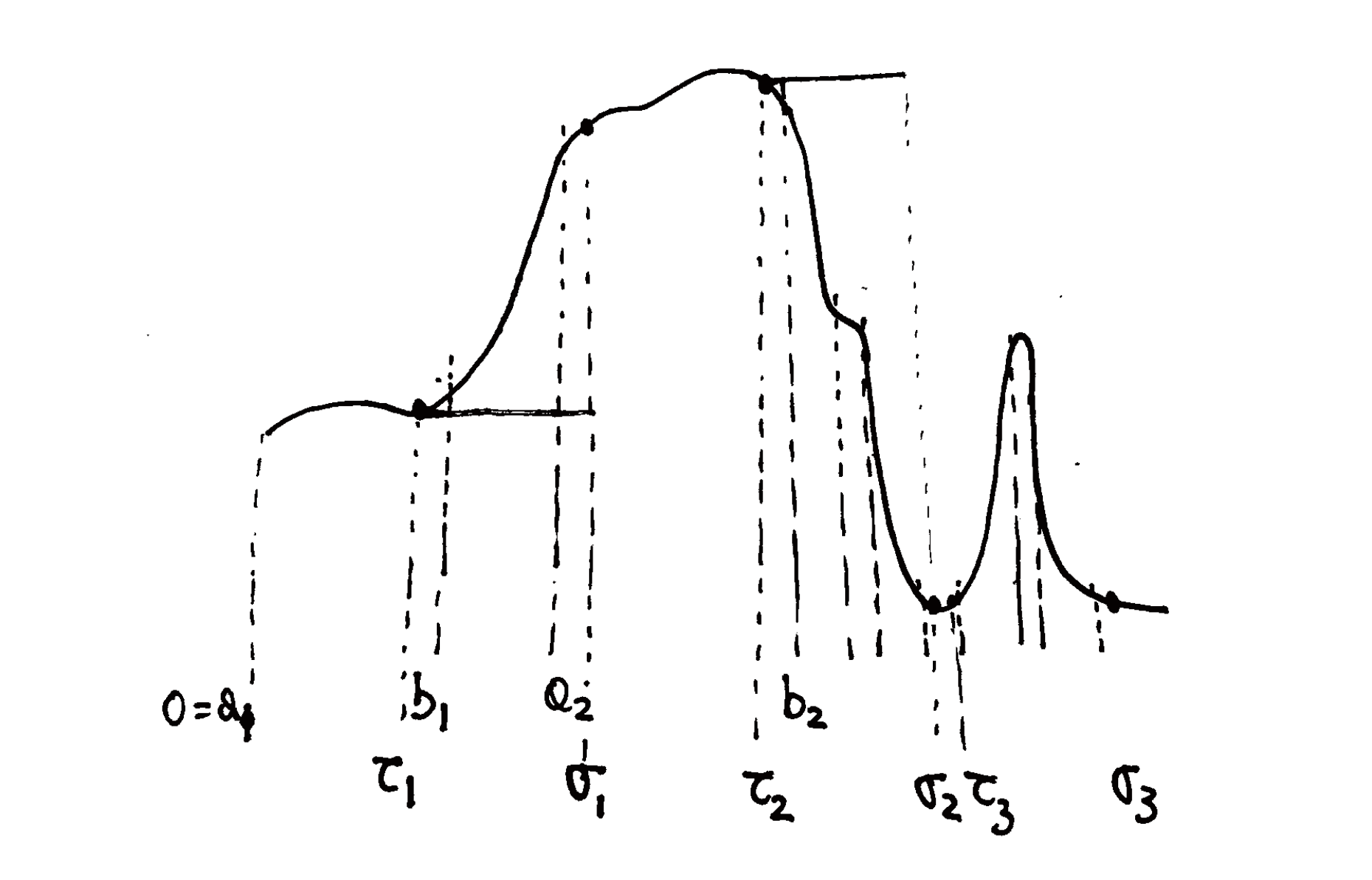}}
\caption{A function $u_\e$ and the related $[\tau,\sigma]$ subdivision}
\label{Fig1}
\end{figure}
Fig.~\ref{Fig1} gives a pictorial representation of a function $u_\e$ and the corresponding subdivision
in intervals $[\tau_n,\sigma_n]=[\tau^\e_n,\sigma^\e_n]$, and of the first step in the construction of $v_\e$; namely, if $|u_\e(\sigma^\e_n)-u_\e(\tau^\e_n)|\ge \theta$, we define 
\begin{equation}
v_\e(t)\coloneq u_\e(\tau^\e_n) \hbox{ for }t\in [\tau^\e_n,\sigma^\e_n),
\end{equation}
as represented in the figure.
Note that, taking $\e>0$ small enough so that the claim of Proposition \ref{sub-prop} is satisfied,
and noting that the pre-factor in \eqref{pre-fa} is also uniformly larger than $(1-\eta)$ for $\e$ small enough, we have
\begin{equation}
F_\e(u_\e, (\tau^\e_n,\sigma^\e_n))\ge (1-\eta)^3 m_k|u_\e(\sigma^\e_j)-u_\e(\tau^\e_j)|^{1/k}.
\end{equation}

The next step is to treat intervals $[\tau^\e_n,\sigma^\e_n]$ such that $|u_\e(\sigma^\e_n)-u_\e(\tau^\e_n)|< \theta$. If $\sigma^\e_n-\tau^\e_n>> \e^{1+\frac{p}{k-1}}$ then we can use Proposition
\ref{stimatausigma1} in \eqref{pre-fa}, to deduce that for $\e$ small enough we have
\begin{equation}
F_\e(u_\e, (\tau^\e_n,\sigma^\e_n))\ge (1-\eta)^2 r\, \psi_{\theta, N} (u_\e(\sigma^\e_n)-u_\e(\tau^\e_n)).
\end{equation}
In this case, again we can substitute $u_\e$ with a constant in this interval, producing a jump point in 
$\sigma^\e_n$.

\begin{figure}[h!]
\centerline{\includegraphics[width=0.4\textwidth]{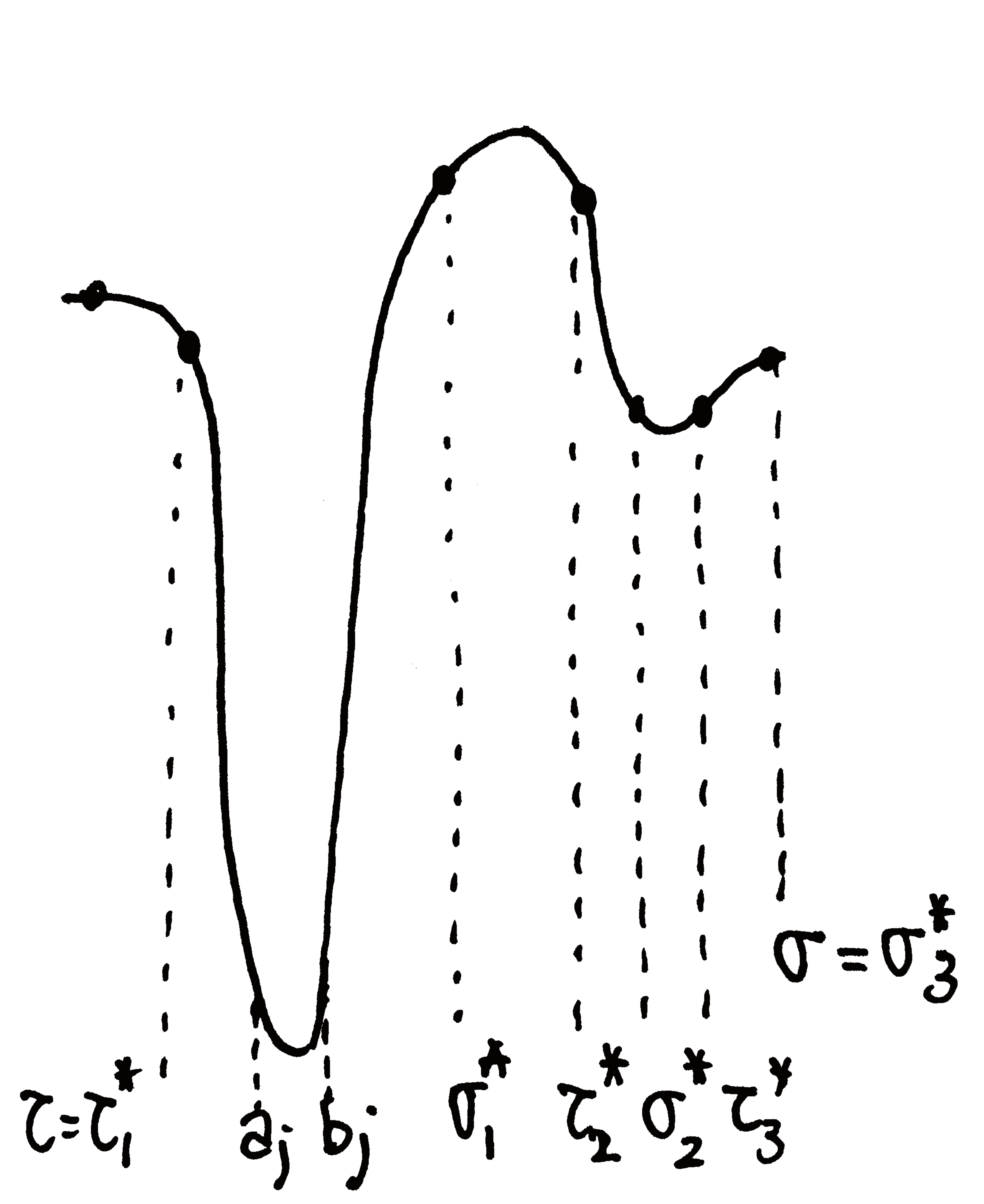}}
\caption{A function $u_\e$ and the related $[\tau^*,\sigma^*]$ subdivision}
\label{Fig2}
\end{figure}
We finally consider the case $\sigma^\e_n-\tau^\e_n\le C\e$ for some constant $C$,
for which the finer subdivision $\mathcal I_\e^*(N)$ must be used as in the previous theorem.
In Fig.~\ref{Fig2} we represent the graph of a function $u_\e$ in an interval $[\tau,\sigma]=[\tau^\e_n,\sigma^\e_n]$ as above. 
 Dropping the dependence on $\e$ and $n$, we denote by $\tau^*_h$ ,$\sigma^*_h$, for $h\in\{1,\ldots, H^n_\e\}$, the endpoints of the disjoint intervals whose union is $[\tau,\sigma]\setminus \bigcup_{j\in \mathcal I^*_\e(N)}\mathcal (a_{*,j}^{\e,N},b_{*,j}^{\e,N})$. 
 Again, we parameterize such intervals increasingly with $h$. 
  In the figure, we have highlighted that some interval $I_j=(a_j,b_j)$ with $j\not \in\mathcal I_\e^*(N)$ may be strictly contained in some $[\tau^*_n,\sigma^*_n]$.
  
  We can estimate
\begin{eqnarray*}
\Big|u_\e(\sigma^\e_n)-u_\e(\tau^\e_n)-\sum_{h=1}^{H^n_\e}(u_\e(\sigma^*_h)-u_\e(\tau^*_h))\Big|
&\le&\Big|\sum_{h=1}^{H^n_\e-1} \int_{(\sigma^*_h, \tau^*_{h+1})}u'_\e dt\Big|\\
&\le&\frac{c_\e}\e(\sigma^\e_n-\tau^\e_n)\le C c_\e.
\end{eqnarray*} 
In particular, for $\e$ small enough we have 
$$
\Big|\sum_{h=1}^{H^n_\e}(u_\e(\sigma^*_h)-u_\e(\tau^*_h))\Big|\le 2\theta.
$$
In the interval $[\tau^\e_n,\sigma^\e_n)$ we define the function $v_\e$ by setting $v_\e(\tau^\e_n)\coloneq u_\e(\tau^\e_n)$, and 
$$
v'_\e(t)\coloneq\begin{cases} 0 &\hbox{ if } \displaystyle t\in \bigcup_{h=1}^{H^n_\e}[\tau^*_h, \sigma^*_h]\\
u'_\e(t) &\hbox{ otherwise.}\end{cases}
$$
In this way, we have
$$
u_\e(\sigma^\e_n+)-v_\e(\sigma^\e_n+)= \sum_{h=1}^{H^n_\e}(u_\e(\sigma^*_h)-u_\e(\tau^*_h)).
$$
 Note that
 \begin{eqnarray}\label{esti-2}\nonumber
\sum_{h=1}^{H^n_\e} F_\e(u_\e,(\tau^*_h, \sigma^*_h))&\ge &(1-\eta)\sum_{h=1}^{H^n_\e} \Big(\psi_{\theta,N} (u_\e(\sigma^*_h)-u_\e(\tau^*_h))-r\frac{|I_{r,h}^*(\e)|}{\e}\Big)
\\
&\ge&  (1-\eta)\psi_{\theta,N}\Big(\sum_{h=1}^{H^n_\e}(u_\e(\sigma^*_h)-u_\e(\tau^*_h))\Big)-
\sum_{h=1}^{H^n_\e}\frac{|I_{r,h}^*(\e)|}{\e},
\end{eqnarray}
where $I_{r,h}^*(\e)$ is given as in \eqref{Ire} with $[\tau^*_h,\sigma^*_h]$ in the place of $[b_{*,j}^{\e,N}, a_{*,j'}^{\e,N}]$.

We finally define $v_\e$ as in Step (i) of the previous proof in the complementary of the sets considered above.

Recalling that the last term in \eqref{esti-2} is negligible by Lemma \ref{lemmar} and Remark \eqref{eqrbordo},
even after summation over all intervals (see Section \ref{estt-2}),
we can rewrite our estimates as
\begin{equation}
\liminf_{\e\to 0} F_\e(u_\e)\ge (1-\eta)^3\liminf_{\e\to 0} \widetilde\Psi_{\theta, N}(v_\e),
\end{equation}
where
$$
 \widetilde\Psi_{\theta, N}(v)\coloneq\int_{(0,1)}|v'|^2dt+\sum_{t\in S(v)}  \widetilde\psi_{\theta, N}(v_\e(t+)-v_\e(t-)),
$$
and
\begin{equation}
 \widetilde\psi_{\theta, N}(z)=\begin{cases}
 r \min\{m(N)\theta^{\frac1k-1} |z|, m({\theta N})|z|^{1/k}\} &\hbox{ if }|z|\le 2\theta\\
 m_k|z|^{1/k} &\hbox{ if }|z|> 2\theta.
 \end{cases}
\end{equation}

We can take $\theta\in (0,\frac12)$ and estimate
$$
 \widetilde\psi_{\theta, N}(z)\ge \min\{r\,m(\theta N)|2\theta|^{1-\frac1k}|z|, m_k|z|^{1/k}\}
$$
We can then proceed as in Remark \ref{som-1}, optimizing the estimates for $\theta\in (0,\frac12)$ and
$N\in\mathbb N$, and noting that
$$
\sup_{\theta, N}\min\{r\,m(\theta N)|2\theta|^{1-\frac1k}|z|, m_k|z|^{1/k}\}= m_k|z|^{1/k},
$$
obtaining the claim.
\end{proof}

It remains to prove the upper bound. This will be done by a density argument, whose first step is an upper bound for a dense family of functions as in the following proposition.

\begin{proposition}\label{ubck} Let $u$ be a piecewise-$C^k$ function such that there exists $\eta>0$ such that for all $t\in S(u)$, $u$ is constant on $(t-\eta,t)$ and $(t, t+\eta)$. Let $(\overline T,\overline v)$ denote the minimal pair for $m_k$, let $z_t\coloneq  u(t^+)-u(t^-)$, and let $\e$ be such that  $|z_t|^{1/k}\overline T\e\le\eta$. If $u_\e$ is defined as
\begin{equation}\label{glisuco}
u_\e(x)\coloneq \begin{cases} u(t^-)+z_t\,\displaystyle\overline v\Big( \frac{x-t}{\e|z_t|^{1/k}}\Big) &\hbox{ if $t\in S(u)$ and  $x\in (t,t+|z_t|^{1/k}\overline T\e)$}\\
u(x) &\hbox{ otherwise,}
\end{cases}\end{equation}
then $u\in H^k(0,1)$, $u_\e\to u$ in $L^1(0,1)$, and 
\begin{equation}\label{ub-eq}
\limsup_{\e\to 0^+} F_\e(u_\e)\le \int_0^1|u'|^2dx +m_k\sum_{t\in S(u)} |u(t^+)-u(t^-)|^{1/k}.
\end{equation}
\end{proposition}

\begin{proof} The convergence $u_\e\to u$ in $L^1(0,1)$ is by construction, since the functions $u_\e$ are bounded in $L^1$ and are equal to $u$ except for a set of vanishing measure. 

For all $t\in S(u)$ we have 
\begin{eqnarray*}
F_\e(u_\e,(t,t+|z_t|^{1/k}\overline T\e))&=& 
\int_0^{|z_t|^{1/k}\overline T\e} \frac{1}{\e|\log\e|} \log\Big(1+\e|\log\e|\Big|\frac{z_t}{\e|z_t|^{1/k}}\Big|^2|v'\Big( \frac{x}{\e|z_t|^{1/k}}\Big)|^2\Big)\dx\\
&&+ \e^{2k-1} \int_0^{|z_t|^{1/k}\overline T\e}\frac{1}{\e^{2k}} \Big|\overline v^{(k)}\Big( \frac{x-t}{\e|z_t|^{1/k}}\Big)\Big|^2\dx\\
&=& 
\int_0^{\overline T} \frac{\e |z_t|^{1/k}}{\e|\log\e|} \log\Big(1+\e|\log\e|\Big|\frac{z_t}{\e|z_t|^{1/k}}\Big|^2|v'(y)|^2\Big)dy\\
&&+ \e^{2k-1}\e |z_t|^{1/k} \int_0^{\overline T}\Big|\frac{z_t}{\e|z_t|^{1/k}}\Big|^{2k} |\overline v^{(k)}(y)|^2 dy\\
&=& 
|z_t|^{1/k}\Big(\int_0^{\overline T} \frac{1}{|\log\e|} \log\Big(1+\frac{|\log\e|}\e\Big|\frac{z_t}{|z_t|^{1/k}}\Big|^2|v'(y)|^2\Big)dy\\
&&\hskip6cm+ \int_0^{\overline T} |\overline v^{(k)}(y)|^2 dy\Big)\\
&\le& 
|z_t|^{1/k}\Big(\overline T \frac{1}{|\log\e|} \log\Big(1+\frac{|\log\e|}\e C\Big)+ \int_0^{\overline T} |\overline v^{(k)}(y)|^2 dy\Big),
\end{eqnarray*}
where $C\coloneq \big|\frac{z_t}{|z_t|^{1/k}}\big|^2\|v'\|_\infty^2$. Now, since $\lim\limits_{\e\to 0}\frac{1}{|\log\e|} \log\big(1+\frac{|\log\e|}\e C\big)=1$, we obtain that
$$
\limsup_{\e\to 0^+} F_\e(u_\e(t,t+\eta))= \limsup_{\e\to 0^+} F_\e(u_\e(t,t+|z_t|^{1/k}\overline T\e))\le m_k |u(t^+)-u(t^-)|^{1/k}
$$
for all $t\in S(u)$.

We now write $S(u)\cup \{-\eta,1\}=\{t_0, t_1, \ldots, t_M\}$  with $t_{j-1}<t_j$. We obtain 
\begin{eqnarray*}
&&\hskip-1.5cm\lim_{\e\to 0}F_\e(u_\e,(t_{j-1}+\eta, t_j))
\\
&=&\lim_{\e\to 0} \Big(
\int_{t_{j-1}+\eta}^{t_j}\frac{1}{\e|\log\e|} \log(1+\e|\log\e||u'|^2)\dx + \e^{2k-1} \int_{t_{j-1}+\eta}^{t_j} |u^{(k)}|^2\dx\Big)
\\
&=& \lim_{\e\to 0} 
\int_{t_{j-1}+\eta}^{t_j}\frac{1}{\e|\log\e|} \log(1+\e|\log\e||u'|^2)\dx
\\
&=& 
\int_{t_{j-1}+\eta}^{t_j}|u'|^2\dx,
\end{eqnarray*}
where the last equality follows from the Dominated Convergence Theorem.
Gathering these equalities valid for all for $j$, and the inequalities valid for $t\in S(u)$, we obtain the claim. 
\end{proof}

Finally, the next proposition completes the proof of the upper bound, and then of Theorem \ref{main} in dimension one.

\begin{proposition}[upper bound]\label{ubb} 
Let $u\in SBV(0,1)$ with $u'\in L^2(0,1)$ and such that $\sum_{t\in S(u)} |u(t^+)-u(t^-)|^{1/k}<+\infty$.
Then there exists a sequence $\{u_\e\}$ that converges to $u$ in $L^1(0,1)$ such that \eqref{ub-eq} holds.
\end{proposition}

\begin{proof} It suffices to remark that the functions satisfying the assumptions of Proposition \ref{ubck} are dense in energy; that is, there exists a sequence $\{u^h\}$ of such functions such that
\begin{eqnarray*}
&&\limsup_{h\to+\infty}\Big( \int_0^1|(u^h)'|^2dx +m_k\sum_{t\in S(u^h)} |u^h(t^+)-u^h(t^-)|^{1/k}\Big)\\
&&\hskip3cm= \int_0^1|u'|^2dx +m_k\sum_{t\in S(u)} |u(t^+)-u(t^-)|^{1/k}.
\end{eqnarray*}
This can be proven, first by writing the function $u\coloneq  w+\sum_{t\in S(u)} (u(t^+)-u(t^-))\chi_{(t,1)}$, where $w\in H^1(0,1)$ and approximating $u$ with $u_n\coloneq w+\sum_{t\in S_n(u)} (u(t^+)-u(t^-))\chi_{(t,1)}$, where $S_n(u)\coloneq \{t\in S(u): |u(t^+)-u(t^-)|>\frac1n\}$ is now a finite set, then extending $u_n$ as a constant in the intervals $(t-2\eta, t)$ and $(t, t+2\eta)$, enlarging the interval of definition of $u_n$ to $(0,1+ 4\eta\#(S_n(u)))$, so
that these new $u_n^\eta$ converge to $u_n$ as $\eta\to0$, and finally mollifying $u^\eta_n$ with a sequence of mollifiers with compact support. 

Since for every such $u^h$ the sequence $\{u^h_\e\}$ constructed in Proposition \ref{ubck} satisfies \eqref{ub-eq}, we can conclude by a diagonal argument.
\end{proof}

\begin{remark}[scaling of the $\Gamma$-limit]\label{scal}\rm
With fixed $\kappa>0$ we can show that the $\Gamma$-limit of the scaled functionals
\begin{equation}\label{PMec}
F^\kappa_\e(u)\coloneq \int_0^1 \frac{\alpha}{\e|\log\e|} \log(1+c\kappa^2\e|\log\e||u'|^2)\dx+ \e^{2k-1} \int_0^1 |u^{(k)}|^2\dx
\end{equation}
is
\begin{equation}\label{PMc}
F^\kappa(u)=\alpha\kappa^2\int_0^1 |u'|^2\dx+\alpha^{1-\frac1{2k}}m_k\sum_{t\in S(u)}|u^+(t)-u^-(t)|^{1/k}.
\end{equation}

Changing variables $v\coloneq \kappa u$, we see that the $\Gamma$-limit of 
$$
G^\kappa_\e(v)=\alpha \Big(\int_0^1 \frac{1}{\e|\log\e|} \log(1+\e|\log\e||v'|^2)\dx+ \e^{2k-1}\frac1{\alpha\kappa^2} \int_0^1 |v^{(k)}|^2\dx\Big)
$$
is given by
\begin{equation}\label{PMcv}
G^\kappa(v)\coloneq \alpha \int_0^1 |v'|^2\dx+\alpha m^{\alpha, \kappa}_k\sum_{t\in S(u)}|v^+(t)-v^-(t)|^{1/k},
\end{equation}
where
\begin{eqnarray*}
&&m^{\alpha, \kappa}_k\coloneq \inf_{T>0}\min\bigg\{T+\frac1{\alpha\kappa^2} \int_0^T|v^{(k)}|^2\,dt: v\in H^k(0,T),  v(0)=0, v(T)=1, \\ &&\hskip4cm v^{(\ell)}(T)=v^{(\ell)}(0)=0 \hbox{ for all } \ell\in\{1,\ldots,k-1\}\Big\}.
\end{eqnarray*}
Indeed all our arguments in the determination of the jump energy density remain unchanged if a constant multiplies the $k$-th derivative. We now note that, by the homogeneity property of the jump-energy formula, we have 
\begin{eqnarray*}
m^{\alpha, \kappa}_k&=&\inf_{T>0}\min\bigg\{T+\int_0^T|w^{(k)}|^2\,dt: v\in H^k(0,T),  w(0)=0, w(T)=\frac1{\kappa\sqrt\alpha}, \\ &&\hskip3cm  w^{(\ell)}(T)=w^{(\ell)}(0)=0 \hbox{ for all } \ell\in\{1,\ldots,k-1\}\Big\}\\
&=&m_k\frac1{\kappa^{1/k}\alpha^{1/2k}}.\end{eqnarray*}
Changing back the variables in \eqref{PMcv}, we then obtain  \eqref{PMc}.
\end{remark}

\section{Surface scaling of the Perona--Malik energies}\label{sufi}
From the above result, we can obtain the $\Gamma$-limit of a different scaling of the Perona--Malik energies, which has been useful to explain some staircasing phenomena in Image Processing \cite{GP-2} when $k=2$. We only state the one-dimensional version, the general case being the analog of Theorem \ref{main}.

\begin{theorem} Let $\mathbb F_\e$ be defined in $H^k(0,1)$ by
\begin{equation}\label{GPMe}
\mathbb F_\e(u)\coloneq \int_0^1 \frac{1}{2\e|\log\e|} \log(1+|u'|^2)\dx+ \e^{2k-1} \int_0^1 |u^{(k)}|^2\dx.
\end{equation}
Then the $\Gamma$-limit of $\mathbb F_\e$ with respect to the $L^1$-convergence is
\begin{equation}\label{GPM}
\mathbb F(u)\coloneq m_k\sum_{t\in S(u)}|u^+(t)-u^-(t)|^{1/k}
\end{equation}
defined on $SBV(0,1)$-functions with $u'=0$ almost everywhere, and where $m_k$ is as in 
\eqref{emmeka-2}.
\end{theorem}

We note that by the monotonicity of $z\mapsto \log(1+z^2)$ we have $\mathbb F_\e\ge F_\e$ as soon as  $\e |\log\e|\le 1$, which implies that $\mathbb F_\e$ enjoy the coerciveness properties of $F_\e$.

\begin{proof} 
The argument just illustrated also shows that for all $\kappa>0$ we have $\mathbb F_\e\ge F^\kappa_\e$ for $\e$ small enough, where $F^\kappa_\e$ is the functional in \eqref{PMec}. Hence, by Remark \ref{scal} we have the lower bounds
$$
\Gamma\hbox{-}\liminf_{\e\to 0}\mathbb F_\e(u) \ge \frac12\kappa^2\int_0^1 |u'|^2\dx+2^{\frac{1}{2k}-1}m_k\sum_{t\in S(u)}|u^+(t)-u^-(t)|^{1/k}
$$
for all $\kappa>0$. By the arbitrariness of $\kappa$ and applying the sup-of-measures lemma, we obtain that the domain of the $\Gamma$-limit are $SBV$ functions $u$ with $u'=0$ and $\sum_{t\in S(u)}|u^+(t)-u^-(t)|^{1/k}
<+\infty$. 

In order to obtain a sharp inequality for the jump term, solving the inequality $\frac12\log(1+z^2)\ge \log(1+\e|\log\e|z^2)$, we note that by Proposition \ref{sub-prop} we have
$$
\frac{1}{2\e|\log\e|}\log (1+z^2)\ge (1-\eta)\max\Big\{\e^{1-2p_\e} z^2,\frac1\e\Big\},
$$
which is the condition that allows one to estimate the jump energy density by $m_k|u(t^+)-u(t^-)|^{1/k}$. Hence, we have obtained the desired lower bound.

As for the upper bound, by density it suffices to check it when $u$ is piecewise constant. For such a $u$ we use the same construction as in \eqref{glisuco}. The argument used in the proof of Proposition \ref{ubck} now becomes
\begin{eqnarray*}&&
\hskip-2cm F_\e(u_\e,(t,t+|z_t|^{1/k}\overline T\e))
\ =\ 
\int_0^{|z_t|^{1/k}\overline T\e} \frac{1}{2\e|\log\e|} \log\Big(1+\Big|\frac{z_t}{\e|z_t|^{1/k}}\Big|^2|v'\Big( \frac{x}{\e|z_t|^{1/k}}\Big)|^2\Big)\dx\\
&&\hskip 3cm + \e^{2k-1} \int_0^{|z_t|^{1/k}\overline T\e}\frac{1}{\e^{2k}} \Big|\overline v^{(k)}\Big( \frac{x-t}{\e|z_t|^{1/k}}\Big)\Big|^2\dx\\
&=& 
|z_t|^{1/k}\Big(\int_0^{\overline T} \frac{1}{2|\log\e|} \log\Big(1+\frac{1}{\e^2}\Big|\frac{z_t}{|z_t|^{1/k}}\Big|^2|v'(y)|^2\Big)dy+ \int_0^{\overline T} |\overline v^{(k)}(y)|^2 dy\Big)\\
&\le& 
|z_t|^{1/k}\Big(\overline T \frac{1}{2|\log\e|} \log\Big(1+\frac{1}{\e^2} C\Big)+ \int_0^{\overline T} |\overline v^{(k)}(y)|^2 dy\Big),
\end{eqnarray*}
which gives the desired result as $\e\to 0$. 
\end{proof}

\noindent \textsc{Acknowledgements.}
The authors gratefully acknowledge the hospitality of the Center of Nonlinear Analysis at Carnegie Mellon University, Pittsburgh. Andrea Braides is member of GNAMPA of INdAM, and is partially supported by the MIUR Excellence Department Project 2023-2027 MatMod@TOV awarded to the Department of Mathematics, University of Rome Tor Vergata.
 The research of Irene Fonseca was partially supported by the National Science Foundation (NSF) under grants DMS--2108784, DMS--2205627 and DMS--2342349.

\goodbreak
 \bibliographystyle{abbrv}
\bibliography{references}

\end{document}